\documentclass[12pt]{article}
\usepackage{amssymb}
\usepackage{mathrsfs}
\usepackage[hypertex]{hyperref}
\usepackage{amsfonts}
\usepackage{graphicx}
\usepackage{latexsym,amsmath,amssymb,amsfonts,amsthm}

\newcommand{\bcen}{\begin{center}}
\newcommand{\ecen}{\end{center}}
\newtheorem{theorem}{Theorem}[section]

\newtheorem{corollary}[theorem]{Corollary}

\newtheorem{remark}[theorem]{Remark}

\begin{document}
\setcounter{page}{1}
\title{A generalized sphere theorem and its applications}
\author{Jing Mao\\
\small{In memory of my father Mr. Xu-Gui Mao}}

\date{}
\protect \footnotetext{\!\!\!\!\!\!\!\!\!\!\!\!{~MSC 2020:
58C40, 35P15.} }
\maketitle ~~~\\[-15mm]

\begin{center}
{\footnotesize Faculty of Mathematics and Statistics,\\
 Key Laboratory of Applied
Mathematics of Hubei Province, \\
Hubei University, Wuhan 430062, China\\
Key Laboratory of Intelligent Sensing System and Security (Hubei
University), Ministry of Education\\
Email: jiner120@163.com
 }
\end{center}


\begin{abstract}
In this paper, we establish \emph{a generalized sphere theorem} for
compact Riemannian manifolds with radial Ricci curvature bounded.
 \end{abstract}


\section{Introduction and our main results}
\renewcommand{\thesection}{\arabic{section}}
\renewcommand{\theequation}{\thesection.\arabic{equation}}
\setcounter{equation}{0}

Along with Euclidean spaces, spheres have attracted considerable
attention. A natural question is:
\begin{itemize}
\item \emph{Under what kind of geometric assumptions, a compact manifold is diffeomorphic (or homeomorphic, isometric) to a sphere of the same dimension?}
\end{itemize}
In the literature, conclusions to the above question were iconically
called ``\emph{sphere theorems}". Many important sphere theorems
have been obtained. For instance, the classical Toponogov sphere
theorem \cite{TVA} and one of its generalizations given by S. Y.
Cheng \cite{CSY} (see also the end of this section for their
detailed statements); Grove-Shiohama's generalized sphere theorem
\cite{GS} which states that \emph{any connected, complete Riemannian
manifold whose sectional curvature $\mathrm{Sec}(\cdot)$ and
diameter $d$ satisfy $\mathrm{Sec}(\cdot)\geq\delta$ and
$d>\pi/2\sqrt{\delta}$ for some $\delta>0$ is homeomorphic to the
sphere}; C. Y. Xia's homeomorphism sphere theorem \cite{XCY};
Brendle-Schoen's (1/4-pinched curvature) differentiable sphere
theorem \cite{BS}. For more information, readers can check \cite{BS,
CSY, GS, XCY} and references therein.

The purpose of this paper is (from the viewpoint of Spectral
Geometry) to establish a rigidity result for compact manifolds under
suitable assumptions such that they are isometric to a class of
spherically symmetric manifolds. Moreover, if the curvature
assumption was strengthened (to a certain extent), this rigidity
result would imply several sphere theorems directly. Actually, this
is the reason why we call the rigidity result (i.e. Theorem
\ref{theo-1}) in this paper \emph{a generalized sphere theorem}.

Given an $n$-dimensional ($n\geq2$) compact Riemannian manifold
$M^{n}$ without boundary, in order to state our main conclusions
clearly, we wish to make the following three
assumptions:\footnote{~In Remarks \ref{remark-1}-\ref{remark-3}
below, we explain that \textbf{\emph{Assumptions 1-3}} are
reasonable and feasible.}

\begin{itemize}

\item (\textbf{\emph{Assumption 1}}) For a point $p\in M^{n}$, there
exists a point $q\in M^{n}$ such that the diameter
$\mathrm{diam}_{M^{n}}$ of the compact manifold $M^{n}$ satisfies
$\mathrm{diam}_{M^{n}}=d(p,q)=l$, where $d(p,q)$ stands for the
Riemannian distance between $p$ and $q$ (that is to say, there
exists a minimizing geodesic joining $p$ and $q$).

\item (\textbf{\emph{Assumption 2}}) $M^{n}$ has a radial Ricci curvature lower bound $(n-1)k(t)$ w.r.t.
the point $p\in M^{n}$, where $t:=d(p,\cdot)$ denotes the Riemannian
distance starting from $p$, that is to say, the radial component of
the Ricci tensor satisfies $\mathrm{Ric}(\gamma'(t),\gamma'(t))\geq
(n-1)k(t)$, with $\gamma(t)$ a unit-speed minimizing geodesic issued
from $p$. Moreover, $k(t)$ is a continuous function defined on
$(0,l)$ having the symmetric property
\begin{eqnarray*}
k(t)=k(l-t),\quad \forall t\in(0,l/2),
\end{eqnarray*}
and letting the following system
\begin{eqnarray} \label{warpf}
\left\{
\begin{array}{lll}
f''(t)+k(t)f(t)=0 \qquad &\mathrm{in}~(0,l),  \\[0.5mm]
f(0)=0,~f'(0)=1,~f(l)=0, \\[0.5mm]
 f|_{(0,l)}>0
\end{array} \right.
\end{eqnarray}
be solvable.

\item (\textbf{\emph{Assumption 3}}) The first nonzero closed
eigenvalue $\Lambda_{1}(M^{n})$ of the Laplacian on $M^{n}$
satisfies $\Lambda_{1}(M^{n})\geq \Lambda^{+}$, where
$\Lambda^{+}>0$ is the positive constant $\Lambda=\Lambda^{+}$
corresponding to the solution $\varphi=\varphi(t)$ of the system
\begin{eqnarray} \label{firste}
\left\{
\begin{array}{ll}
\frac{d^{2}\varphi(t)}{dt^{2}}+(n-1)\frac{f'(t)}{f(t)}\frac{d\varphi(t)}{dt}+\Lambda\cdot\varphi(t)=0 \qquad &\mathrm{in}~\left(0,\frac{l}{2}\right),  \\[1mm]
\varphi'(0)=0,~\varphi\left(\frac{l}{2}\right)=0 \\[0.5mm]
 \varphi|_{(0,l/2)}>0,
\end{array} \right.
\end{eqnarray}
with $f(t)$ the solution to the system (\ref{warpf}).

\end{itemize}

In fact, we can prove:

\begin{theorem}  \label{theo-1}
Suppose that $M^{n}$ is an $n$-dimensional compact Riemannian
manifold (without bondary) satisfying the diameter
\textbf{Assumption 1}, the curvature \textbf{Assumption 2}, and the
eigenvalue lower bound \textbf{Assumption 3}. Then $M^{n}$ is
isometric to a one-point compactification metric space
$\overline{M^{\ast}}:=M^{\ast}\cup\left(\{l\}\times_{f(l)}\mathbb{S}^{n-1}\right)$,
where $M^{\ast}:=[0,l)\times_{f}\mathbb{S}^{n-1}$ is a spherically
symmetric manifold
 endowed with a one-point
compactification topology at the point
$\{l\}\times_{f(l)}\mathbb{S}^{n-1}$, and the warping function $f$
is determined by the system (\ref{warpf}).
\end{theorem}

\begin{remark}
\rm{ Since $f$ is determined by the system (\ref{warpf}), which
implies $f(l)=0$, the set $\{l\}\times_{f(l)}\mathbb{S}^{n-1}$ must
degenerate into a single point $\{l\}\times_{0}\mathbb{S}^{n-1}$,
with $\mathbb{S}^{n-1}$ the unit Euclidean $(n-1)$-sphere. Hence,
$\overline{M^{\ast}}$ should be closed if a one-point
compactification topology at the point
$\{l\}\times_{0}\mathbb{S}^{n-1}$ was imposed.}
\end{remark}

\begin{remark} \label{remark-1}
\rm{ For a given complete Riemannian $n$-manifold $M^{n}$, $n\geq2$,
the notion that a domain $\Omega\subset M^{n}$ is said to be
spherically symmetric can be well-defined (see e.g. \cite[Definition
2.1]{FMI}). For readers' convenience, we wish to repeat it here as
follows:
\begin{itemize}
\item (\cite[Definition
2.1]{FMI}) \emph{A domain $\Omega=\exp_p([0,l)\times{S}_p^{n-1})
\subset M^{n}\backslash \mathrm{Cut}(p)$, with $l<inj(p)$,  is said
to be spherically symmetric with respect to a point $p\in \Omega$,
if
 and only if
the matrix $\mathbb{A}(t,\xi)$ satisfies $\mathbb{A}(t,\xi)=f(t)I$,
for a function $f\in{C^{2}([0,l))}$,  with   $f(0)=0$, $f'(0)=1$,
and  $f|(0,l)>0$.}
\end{itemize}
 Here $\exp_p$ denotes the exponential mapping at the point $p\in
 M^{n}$, $\mathrm{Cut}(p)$ and $inj(p)$ stand for the cut-locus of $p$ and the injectivity radius at
 $p$ respectively, ${S}_p^{n-1}$ is the unit sphere in the tangent
 space $T_{p}M^{n}$ at $p$, $I$ denotes the identity matrix, and $\mathbb{A}(t,\xi)$ is the path of linear
 transformations defined as in \cite[p. 703]{FMI}.
  We strongly
 suggest readers to check \cite[Section 2]{FMI} carefully not only
 for the above definition of being spherically symmetric but also
 for strict definitions that manifolds have radial (Ricci or
 sectional) curvature (lower  or upper) bound with respect to a
 given point (see \cite[Definitions 2.2 and 2.3]{FMI}), a spectral
 asymptotical property of spherically symmetric manifolds (see \cite[Lemma
 2.5]{FMI}), and some other fundamental properties. We wish to
 mention that readers can also get these facts in the author's other works \cite{JM3,
 MDW}.

 In fact, for a fixed unit vector $\xi\in S^{n-1}_{p}\subset T_{p}M^n$, let
$\xi^{\bot}$ be the orthogonal complement of $\{\mathbb{R}\xi\}$ in
$T_{p}M^n$ and $\tau_{t}:
T_{p}M^n\rightarrow{T_{\exp_{p}(t\xi)}M^n}$ be the parallel
translation along $\gamma_{\xi}:=\exp_{p}(t\xi)$. The path of linear
transformations
$\mathbb{A}(t,\xi):\xi^{\bot}\rightarrow{\xi^{\bot}}$ is given by
 \begin{eqnarray*}
\mathbb{A}(t,\xi)\eta=(\tau_{t})^{-1}Y_{\eta}(t),
 \end{eqnarray*}
where $Y_{\eta}(t)=d(\exp_p)_{(t\xi)}(t\eta)$ is the Jacobi field
along $\gamma_{\xi}$ satisfying $Y_{\eta}(0)=0$ and
$(\nabla_{t}Y_{\eta})(0)=\eta$. This operator satisfies the Jacobi
equation
 $\mathbb{A}''+\mathcal{R}\mathbb{A}=0$ with initial conditions
 $\mathbb{A}(0,\xi)=0$, $\mathbb{A}'(0,\xi)=I$, where
$\mathcal{R}(t)$ is the self-adjoint operator on $\xi^{\bot}$, $
\mathcal{R}(t)\eta=(\tau_{t})^{-1}R(\gamma'_{\xi}(t),\tau_{t}\eta)
\gamma'_{\xi}(t)$ with $R$ the curvature tensor on $M^n$. Besides,
for the complete Riemannian $n$-manifold $M^{n}$, $n\geq2$, one has
the following notion:
\begin{itemize}
\item (see e.g. \cite[Definition 2.2]{FMI}) \emph{Given a continuous function
$k:[0,l)\rightarrow \mathbb{R}$, we say that $M^{n}$ has a radial
Ricci curvature lower bound $(n-1)k$ at the point $p$ if
\begin{eqnarray} \label{def-cur1}
\mathrm{Ric}(v_x,v_x)\geq(n-1)k(r(x)), \quad\forall x\in
M^{n}\backslash Cut(p)\cup \{p\} ,
\end{eqnarray}
where $\mathrm{Ric}$ is the Ricci curvature of $M^{n}$.}
\end{itemize}
Here the radial distance is $r(x)=d(x,p)$ for
$x=\gamma_{\xi}(t)=\exp_{p}(t\xi)$, and $v_x=\nabla r(x)$ is the
radial unit tangent vector at $x$, with $\nabla$ the gradient
operator on $M^n$. The parameter $t$ may be seen as the argument of
the continuous function $k:[0,l)\rightarrow\mathbb{R}$, and
therefore, one has $\frac{d}{dt}|_{x}=\nabla r(x)=v_{x}$. Then, the
expression (\ref{def-cur1}) directly becomes $
\mathrm{Ric}(\frac{d}{dt},\frac{d}{dt})\geq(n-1)k(t)$.

 Sometimes, spherically symmetric manifolds are also called \emph{generalized space
 forms} (as named by Katz and Kondo \cite{KK}) and a standard model for such manifolds is given by
the quotient manifold of the warped product
$[0,l)\times_{f}\mathbb{S}^{n-1}$ with the metric
\begin{eqnarray} \label{metric}
 ds^{2}=dt^{2}+f(t)^{2}|d\xi|^{2}, \quad \forall
\xi \in S_p^{n-1},~ 0<t<l.
 \end{eqnarray}
Besides, as shown in \cite[p. 706]{FMI}, the radial sectional
curvature and the radial component of the Ricci tensor of a model
space $[0,l)\times_{f} \mathbb{S}^{n-1}$, with $f$ of class $C^2$,
 are  respectively given
by
\begin{equation*}
\begin{array}{ll}
\mathcal{K}(\frac{d}{dt}, V)=R(\frac{d}{d t},V,\frac{d}{d t}, V)
=-\frac{f''(t)}{f(t)}&\mbox{~for}~~ V\in T_{\xi}\mathbb{S}^{n-1},
~|V|=1,\\
\mathrm{Ric}(\frac{d}{d t},\frac{d}{d t})=-(n-1)\frac{f''(t)}{f(t)},
\end{array}
\end{equation*}
with $R(\cdot,\cdot,\cdot,\cdot)$ the curvature tensor as above.
This fact (together with Theorem \ref{theo-CEC} below) is exactly
the reason why the warping function of the spherically symmetric
manifold $M^{\ast}$ constructed in Theorem \ref{theo-1} is
determined by the system (\ref{warpf}). Especially, a space form
with constant sectional curvature $K$ is certainly a spherically
symmetric manifold and in this particular situation one has
\begin{eqnarray*}
f(t)=\left\{
\begin{array}{llll}
\frac{\sin\sqrt{K}t}{\sqrt{K}}, & \quad  l= \frac{\pi}{\sqrt{K}}
  & \quad K>0,\\
 t, &\quad l=+\infty & \quad K=0, \\
\frac{\sinh\sqrt{-K}t}{\sqrt{-K}}, & \quad l=+\infty  &\quad K<0.
\end{array}
\right.
\end{eqnarray*}
 }
\end{remark}

\begin{remark}  \label{remark-added}
\rm{
  For the spherically symmetric manifold $M^{\ast}$
 constructed in Theorem \ref{theo-1}, since in this situation $l$ is
 finite and $f(l)=0$ (i.e. $M^{\ast}$ ``\emph{closes}"), one needs
 to define a one-point compactification topology at the closing point
$\{l\}\times_{0}\mathbb{S}^{n-1}$ such that at this closing point
the metric (\ref{metric}) can \emph{be extended continuously}, and
then the space $\overline{M^{\ast}}$ would be a Riemannian metric
space. The metric (\ref{metric}) is of class $C^k$, $k\geq0$, if
$f\in C^{k}((0,l))$ and of class $C^{k+3}$ at $t=0$ with vanishing
$2d$-derivatives at $t=0$, for all $2d\leq k+3$ (see p. 13 of
\cite{PP}). Similarly, the metric is of class $C^k$ at the closing
point, if the even derivatives of $f$ satisfy at $t=l$ conditions
analogous to those satisfied at $t=0$. Hence, the fact, i.e.
imposing a one-point compactification topology at the closing point
such that the metric (\ref{metric}) can be extended continuously,
indicates that at $t=l$, $f$ is $C^3$ with $f'(l)=-1$ and
$f''(l)=0$. See also \cite[p. 706]{FMI} for more details. }
\end{remark}

\begin{remark}  \label{remark-2}
\rm{
 We wish to say that there exist many continuous functions $k(t)$
of different types such that the system (\ref{warpf}) is solvable.
For instance,
 \begin{itemize}

\item
 if $k(t)\equiv K>0$ is a constant function,
 $l=\pi/\sqrt{K}$,
then one has
\begin{eqnarray*}
f(t)=\frac{\sin(\sqrt{K}t)}{\sqrt{K}},\qquad \Lambda^{+}=nK,
\end{eqnarray*}
which definitely satisfy \textbf{\emph{Assumptions 1-3}}. By Theorem
\ref{theo-1}, in this setting, the compact manifold $M^{n}$ is
isometric to $\mathbb{S}^{n}(\frac{1}{\sqrt{K}})$, i.e. a Euclidean
$n$-sphere of radius $\frac{1}{\sqrt{K}}$ (or of constant sectional
curvature $K$).

\item if $k(t)=\frac{12}{45-(t-3)^{2}}$, $t\in[0,6)$, $l=6$, then one can
get from (\ref{warpf}) that
\begin{eqnarray} \label{ex-2}
f(t)=\frac{15}{8}-\frac{1}{4}(t-3)^{2}+\frac{1}{216}(t-3)^{4},
\end{eqnarray}
and in this situation, by  Theorem \ref{theo-1}, one knows that
$M^{n}$ is isometric to the spherically symmetric manifold
$[0,6]\times_{f}\mathbb{S}^{n-1}$ with the warping function $f(t)$
given by (\ref{ex-2}).

\item if $k(t)=\frac{12}{80-(t-4)^{2}}$, $t\in[0,8)$, $l=8$, then one can
get from (\ref{warpf}) that
\begin{eqnarray*}
f(t)=\frac{1}{512}\left[(t-4)^{4}-96(t-4)^{2}+1280\right].
\end{eqnarray*}
\end{itemize}
Clearly, different choices of the continuous function $k(t)$ and the
diameter $l$ create different warping functions $f(t)$, and
consequently would give different isometric manifolds
$\overline{M^{\ast}}$. }
\end{remark}

\begin{remark}   \label{remark-3}
\rm{
  For a spherically symmetric manifold
 $[0,l)\times_{f}\mathbb{S}^{n-1}$, generally the point
 $\{0\}\times_{f(0)}\mathbb{S}^{n-1}$ is called the base point of this
 manifold. Denote this base point by $p^{\ast}$. We wish to point out that the system (\ref{firste}) is
solvable, and moreover, $\Lambda^{+}$ is the first Dirichlet
eigenvalue
$\lambda_{1}\left(\mathscr{B}_{M^{\ast}}(p^{\ast},\frac{l}{2})\right)$
of the Laplacian on the geodesic ball
$\mathscr{B}_{M^{\ast}}(p^{\ast},\frac{l}{2})$, with center
$p^{\ast}$ and radius $l/2$, of the spherically symmetric manifold
$M^{\ast}=[0,l)\times_{f}\mathbb{S}^{n-1}$
 constructed in Theorem \ref{theo-1}. In this setting, the solution
 $\varphi(t)$ to the system (\ref{firste}) is the eigenfunction
 belonging to the eigenvalue $\Lambda^{+}$. In fact, by Courant's
 nodal domain theorem for the Dirichlet eigenvalue problem of the
 Laplacian (see e.g. \cite[Chapter I]{IC}), one knows that
 eigenfunctions of
 $\lambda_{1}\left(\mathscr{B}_{M^{\ast}}(p^{\ast},\frac{l}{2})\right)$
 do not change sign on
 $\mathscr{B}_{M^{\ast}}(p^{\ast},\frac{l}{2})$, i.e. the number of the nodal
 domain of its eigenfunctions is $1$ and the multiplicity of $\lambda_{1}\left(\mathscr{B}_{M^{\ast}}(p^{\ast},\frac{l}{2})\right)$
  is also $1$. Besides, for the geodesic
 ball  $\mathscr{B}_{M^{\ast}}(p^{\ast},\frac{l}{2})$ with the
 metric (\ref{metric}), it is easy to see that the Laplacian on
 $\mathscr{B}_{M^{\ast}}(p^{\ast},\frac{l}{2})$ can be rewritten as
  \begin{eqnarray*}
\Delta=
\frac{d^{2}}{dt^{2}}+(n-1)\frac{f'(t)}{f(t)}\frac{d}{dt}+\Delta_{\mathbb{S}^{n-1}},
  \end{eqnarray*}
where $\Delta_{\mathbb{S}^{n-1}}$ denotes the Laplacian on
$\mathbb{S}^{n-1}$ with respect to the round metric. Based on these
facts, it is not hard to get that eigenfunctions of
$\lambda_{1}\left(\mathscr{B}_{M^{\ast}}(p^{\ast},\frac{l}{2})\right)$
are radial and satisfy the system (\ref{firste}). Moreover,
$\Lambda=\Lambda^{+}=\lambda_{1}\left(\mathscr{B}_{M^{\ast}}(p^{\ast},\frac{l}{2})\right)>0$.
The condition $\varphi'(0)=0$ in the system (\ref{firste}) is
imposed to ensure the smoothness of the eigenfunction $\varphi(t)$.
Readers can also check \cite[Lemmas 3.1 and 3.2]{FMI} for an
explanation of facts mentioned here.
 }
\end{remark}

By Theorem \ref{theo-1} and Remark \ref{remark-2}, one easily gets
the following sphere theorem.

\begin{corollary} \label{coro-1}
Suppose that $M^{n}$ is an $n$-dimensional compact Riemannian
manifold whose radial Ricci curvature is bounded from below by some
constant $(n-1)K>0$ w.r.t. $p\in M^{n}$,
$\Lambda_{1}(M^{n})\geq\Lambda^{+}=nK$, and
$\mathrm{diam}_{M^{n}}=d(p,q)=\frac{\pi}{\sqrt{K}}$ for some point
$q\in M^{n}$. Then $M^{n}$ is isometric to
$\mathbb{S}^{n}(\frac{1}{\sqrt{K}})$.
\end{corollary}

The above corollary directly implies:

\begin{corollary} \label{coro-2}
Suppose that $M^{n}$ is an $n$-dimensional compact Riemannian
manifold with radial Ricci curvature bounded from below by some
constant $(n-1)K>0$, $\Lambda_{1}(M^{n})\geq\Lambda^{+}=nK$, and
$\mathrm{diam}_{M^{n}}=\frac{\pi}{\sqrt{K}}$. Then $M^{n}$ is
isometric to $\mathbb{S}^{n}(\frac{1}{\sqrt{K}})$.
\end{corollary}

\begin{remark}
\rm{(1) In Corollary \ref{coro-1}, we only require the existence of
a point $p\in M^{n}$ from which the radial Ricci curvature has a
lower bound $(n-1)K>0$. However, in Corollary \ref{coro-2}, the
positive lower bound assumption for the radial Ricci curvature is
\emph{pointwise}, which of course is stronger than the curvature
assumption
in Corollary \ref{coro-1}. \\
 (2) If the curvature assumption in Corollary \ref{coro-2} was
 strengthened to be ``\emph{the Ricci curvature $\mathrm{Ric}(M^{n})$ of $M^n$ has a lower bound
 $(n-1)K>0$}", then the assumption ``$\Lambda_{1}(M^{n})\geq\Lambda^{+}=nK$" in Corollary \ref{coro-2} can be removed. This is because,
 if $\mathrm{Ric}(M^{n})\geq(n-1)K>0$, A. Lichnerowicz \cite{LA}
 obtained the eigenvalue estimate $\Lambda_{1}(M^{n})\geq nK$ by
 applying Bochner's formula directly.
  \\
 (3) If, moreover, the curvature assumption in Corollary \ref{coro-2} was
 strengthened to be ``$\mathrm{Ric}(M^{n})\geq(n-1)K>0$", then our Corollary \ref{coro-2} degenerates into a
generalized
 Toponogov sphere theorem proven by S. Y. Cheng (see \cite[Theorem
 3.1]{CSY}). That is to say, it holds:
\begin{itemize}
\item \emph{Suppose that $M^{n}$ is an $n$-dimensional compact Riemannian
manifold with Ricci curvature bounded from below by some constant
$(n-1)K>0$, and $\mathrm{diam}_{M^{n}}=\frac{\pi}{\sqrt{K}}$. Then
$M^{n}$ is isometric to $\mathbb{S}^{n}(\frac{1}{\sqrt{K}})$.}
\end{itemize}
 If furthermore the curvature assumption was
 strengthened to be ``\emph{the sectional curvature $\mathrm{Sec}(M^n)$ of $M^n$ has a lower bound
 $K>0$}", then Corollary \ref{coro-2} degenerates into the classical
 Toponogov sphere theorem, which states:
 \begin{itemize}
 \item \emph{For an $n$-dimensional compact Riemannian
manifold $M^n$, if $\mathrm{Sec}(M^n)\geq K>0$ and
$\mathrm{diam}_{M^{n}}=\frac{\pi}{\sqrt{K}}$, then $M^{n}$ is
isometric to $\mathbb{S}^{n}(\frac{1}{\sqrt{K}})$.}
 \end{itemize}
}
\end{remark}

In fact, if $n=2$, the radial Ricci curvature and the Ricci
curvature of the compact surface $M^{n}=M^{2}$ coincide with its
Gaussian curvature, and in this situation, $\Lambda_{1}(M^2)\geq2K$
follows, Corollary \ref{coro-2} becomes exactly Cheng's generalized
 Toponogov sphere theorem \cite[Theorem
 3.1]{CSY}. This fact inspires us to consider:

\vspace{3mm}
 \textbf{Question}. \emph{Is it possible to get the eigenvalue estimate $\Lambda_{1}(M^{n})\geq nK$ if $n\geq3$ and the
 radial Ricci curvature of $M^{n}$ is bounded from below by some
constant $(n-1)K>0$? }

\section{Proof of Theorem \ref{theo-1}}
\renewcommand{\thesection}{\arabic{section}}
\renewcommand{\theequation}{\thesection.\arabic{equation}}
\setcounter{equation}{0}

\begin{theorem} \label{theo-CEC} \cite[Theorem 3.6]{FMI} (Cheng-type eigenvalue comparison theorem)
Let $M^n$ be a complete $n$-dimensional Riemannian manifold with a
radial Ricci curvature
 lower bound $(n-1)k(t)$ with respect to a point $p \in M^n$.
We then have
\begin{eqnarray}  \label{CEC}
\lambda_{1}(B(p,r_{0}))\leq\lambda_{1}\left(\mathscr{B}_{M^{\ast}}(p^{\ast},r_{0})\right),
\end{eqnarray}
where $\lambda_{1}(\cdot)$ denotes the first eigenvalue of the
corresponding geodesic ball. Moreover, the equality in (\ref{CEC})
holds if and only if $B(p,r_{0})$ is isometric to
$\mathscr{B}_{M^{\ast}}(p^{\ast},r_{0})$.
\end{theorem}

\begin{remark}
\rm{In \cite[Section 6]{JM4}, the author has obtained a heat kernel
comparison theorem for complete Riemannian manifolds with radial
(Ricci or sectional) curvature bounded, and then has used this heat
kernel comparison theorem to give a second proof of the above
Cheng-type eigenvalue comparison theorem.
 }
\end{remark}

Using Theorem \ref{theo-CEC} directly, we can get the following
eigenvalue comparison result.

\begin{theorem} \label{theo-2}
Suppose that $M^{n}$ is an $n$-dimensional compact Riemannian
manifold satisfying the diameter \textbf{Assumption 1} and the
curvature \textbf{Assumption 2}. Then
 \begin{eqnarray*}
\Lambda_{1}(M^{n})\leq\lambda_{1}\left(\mathscr{B}_{M^{\ast}}\left(p^{\ast},\frac{l}{2}\right)\right),
 \end{eqnarray*}
where $\mathscr{B}_{M^{\ast}}\left(p^{\ast},\frac{l}{2}\right)$
denotes the geodesic ball, centered at the base point
$p^{\ast}=\{0\}\times_{f(0)}\mathbb{S}^{n-1}$ and of radius $l/2$,
on the spherically symmetric manifold
$M^{\ast}=[0,l)\times_{f}\mathbb{S}^{n-1}$.
\end{theorem}

\begin{proof}
By the diameter \textbf{\emph{Assumption 1}}, one knows that
geodesic balls $B(p,\frac{l}{2})$ and $B(q,\frac{l}{2})$ are
disjoint. Set $\phi_{1}:=\phi(d(p,\cdot))$,
$\phi_{2}:=\phi(d(q,\cdot))$, where $\phi$ is the radial
eigenfunction of the first Dirichlet eigenvalue
$\lambda_{1}\left(\mathscr{B}_{M^{\ast}}\left(p^{\ast},\frac{l}{2}\right)\right)$.
It is easy to check that $\phi_{1}\in W^{1,2}_{0}(B(p,\frac{l}{2}))$
and $\phi_{2}\in W^{1,2}_{0}(B(q,\frac{l}{2}))$. Then the proof of
Theorem \ref{theo-CEC} shown in \cite[pp. 712-713]{FMI} implies
under the curvature \textbf{\emph{Assumption 2}} that
 \begin{eqnarray*}
 \int_{B(p,\frac{l}{2})}(\phi'_{1})^{2}\leq\lambda_{1}\left(\mathscr{B}_{M^{\ast}}\left(p^{\ast},\frac{l}{2}\right)\right)\int_{B(p,\frac{l}{2})}(\phi_{1})^{2}
 \end{eqnarray*}
 and
\begin{eqnarray*}
 \int_{B(q,\frac{l}{2})}(\phi'_{2})^{2}\leq\lambda_{1}\left(\mathscr{B}_{M^{\ast}}\left(p^{\ast},\frac{l}{2}\right)\right)\int_{B(q,\frac{l}{2})}(\phi_{2})^{2}.
 \end{eqnarray*}
Extend $\phi_{1}$ to be zero out of $B(p,\frac{l}{2})$ and
$\phi_{2}$ to be zero out of $B(q,\frac{l}{2})$ such that those
extensions are in the Sobolev space $W^{1,2}(M^n)$. There exist
constants $a_{1}$, $a_{2}$ (not all zero) such that
 \begin{eqnarray*}
 \int_{M^{n}}(a_{1}\phi_{1}+a_{2}\phi_{2})=0.
 \end{eqnarray*}
This is because the above equality is equivalent with the following
equation
 \begin{eqnarray*}
a_{1}\int_{M^{n}}\phi_{1}+a_{2}\int_{M^{n}}\phi_{2}=0,
 \end{eqnarray*}
 and it always has solutions $a_{1}$, $a_{2}$ (not all zero) since
 in this situation
one has two unknown variables and only one equation. In fact,
 $a_{1}$, $a_{2}$ satisfy
\begin{eqnarray*}
\frac{a_1}{a_2}=-\int_{M^{n}}\phi_{2}\bigg{/}\int_{M^{n}}\phi_{1}.
\end{eqnarray*}
 Together with
the fact that $B(p,\frac{l}{2})$ and $B(q,\frac{l}{2})$ are
disjoint, one has $a_{1}\phi_{1}+a_{2}\phi_{2}$ cannot vanish
identically. Therefore, using the variational principle, one can get
 \begin{eqnarray*}
 \Lambda_{1}(M^n)\int_{M^{n}}(a_{1}\phi_{1}+a_{2}\phi_{2})^{2}&\leq&\int_{M^{n}}(a_{1}\phi'_{1}+a_{2}\phi'_{2})^{2}\\
 &=& \int_{M^{n}}\left[(a_{1}\phi_{1}+a_{2}\phi_{2})'\right]^{2}\\
 &\leq&
 \lambda_{1}\left(\mathscr{B}_{M^{\ast}}\left(p^{\ast},\frac{l}{2}\right)\right)\int_{M^{n}}(a_{1}\phi_{1}+a_{2}\phi_{2})^{2},
 \end{eqnarray*}
which implies
$\Lambda_{1}(M^{n})\leq\lambda_{1}\left(\mathscr{B}_{M^{\ast}}\left(p^{\ast},\frac{l}{2}\right)\right)$.
This completes the proof of Theorem \ref{theo-2}.
\end{proof}

Now we can give the proof of our main conclusion:

\begin{proof} [Proof of Theorem \ref{theo-1}]
By the diameter  \textbf{\emph{Assumption 1}}, one knows that
geodesic balls $B(p,\frac{l}{2})$ and $B(q,\frac{l}{2})$ are
disjoint. By the curvature \textbf{\emph{Assumption 2}}, the
eigenvalue lower bound \textbf{\emph{Assumption 3}}, Theorem
\ref{theo-CEC} (i.e. Cheng-type eigenvalue comparison), and Theorem
\ref{theo-2}, we have
 \begin{eqnarray} \label{2-1-add}
 \Lambda^{+}\leq\Lambda_{1}(M^{n})\leq\lambda_{1}\left(B\left(p,\frac{l}{2}\right)\right)\leq
 \lambda_{1}\left(\mathscr{B}_{M^{\ast}}\left(p^{\ast},\frac{l}{2}\right)\right)=
\Lambda^{+}
 \end{eqnarray}
and
\begin{eqnarray} \label{2-2-add}
 \Lambda^{+}\leq\Lambda_{1}(M^{n})\leq\lambda_{1}\left(B\left(q,\frac{l}{2}\right)\right)\leq
 \lambda_{1}\left(\mathscr{B}_{M^{\ast}}\left(p^{\ast},\frac{l}{2}\right)\right)=
 \Lambda^{+},
 \end{eqnarray}
which implies
$\lambda_{1}\left(B\left(p,\frac{l}{2}\right)\right)=\Lambda^{+}$
and
$\lambda_{1}\left(B\left(q,\frac{l}{2}\right)\right)=\Lambda^{+}$.
Consequently, one easily deduces that geodesic balls
$B(p,\frac{l}{2})$, $B(q,\frac{l}{2})$ are isometric to
$\mathscr{B}_{M^{\ast}}\left(p^{\ast},\frac{l}{2}\right)$.

By the way, although the second inequality in (\ref{2-1-add}) or
(\ref{2-2-add}) would not be obtained by using Theorem \ref{theo-2}
directly, it can be achieved by using a similar idea as that in the
proof of Theorem \ref{theo-2} above. In fact, let $\psi_1$, $\psi_2$
be eigenfunctions corresponding to the first Dirichlet eigenvalues
$\lambda_{1}\left(B\left(p,\frac{l}{2}\right)\right)$ and
$\lambda_{1}\left(B\left(q,\frac{l}{2}\right)\right)$ respectively,
and then extend $\psi_{1}$, $\psi_2$ to be zero out of
$B(p,\frac{l}{2})$, $B(q,\frac{l}{2})$ such that their extensions
are in the Sobolev space $W^{1,2}(M^n)$. Similar to the argument at
the end of proof of Theorem \ref{theo-2}, there exist constants
$b_1$, $b_2$ (not all zero) such that
  \begin{eqnarray*}
 \int_{M^{n}}(b_{1}\psi_{1}+b_{2}\psi_{2})=0.
 \end{eqnarray*}
Then one has by using the variational principle that
 \begin{eqnarray} \label{2-3-add}
\Lambda_{1}(M^{n})&\leq&\frac{\int_{M^{n}}(b_{1}\nabla\psi_{1}+b_{2}\nabla\psi_{2})^{2}}{\int_{M^{n}}(b_{1}\psi_{1}+b_{2}\psi_{2})^{2}}\nonumber\\
&\leq&
\frac{b_{1}^{2}\int_{B(p,\frac{l}{2})}(\nabla\psi_{1})^{2}+b_{2}^{2}\int_{B(q,\frac{l}{2})}(\nabla\psi_{2})^{2}}
{b_{1}^{2}\int_{B(p,\frac{l}{2})}\psi_{1}^{2}+b_{2}^{2}\int_{B(q,\frac{l}{2})}\psi_{2}^{2}}\nonumber\\
&\leq&\min\left\{\frac{\int_{B(p,\frac{l}{2})}(\nabla\psi_{1})^{2}}{\int_{B(p,\frac{l}{2})}\psi_{1}^{2}},
\frac{\int_{B(q,\frac{l}{2})}(\nabla\psi_{2})^{2}}{\int_{B(q,\frac{l}{2})}\psi_{2}^{2}}\right\}\nonumber\\
&=&\min\left\{\lambda_{1}\left(B\left(p,\frac{l}{2}\right)\right),\lambda_{1}\left(B\left(q,\frac{l}{2}\right)\right)\right\},
 \end{eqnarray}
 where as before $\nabla$ denotes the gradient operator on $M^n$.

Now, if $B(q,\frac{l}{2})$ is strictly contained in
$M^{n}\setminus\overline{B(p,\frac{l}{2})}$, then using the domain
monotonicity of Dirichlet eigenvalues of the Laplacian (see e.g.
\cite{IC}), one can get
\begin{eqnarray*}
\lambda_{1}\left(M^{n}\setminus
B\left(p,\frac{l}{2}\right)\right)<\lambda_{1}\left(B\left(q,\frac{l}{2}\right)\right)=\Lambda^{+}.
\end{eqnarray*}
Therefore, one has $\Lambda_{1}(M^{n})<\Lambda^{+}$ since
 \begin{eqnarray*}
\Lambda_{1}(M^{n})\leq\lambda_{1}\left(M^{n}\setminus
B\left(p,\frac{l}{2}\right)\right)
 \end{eqnarray*}
 holds by using a similar argument to that of proving (\ref{2-3-add}).
This is contradict with the \textbf{\emph{Assumption 3}} that
$\Lambda_{1}(M^{n})\geq\Lambda^{+}$, and consequently the only
possibility is that
 \begin{eqnarray*}
 M^{n}=\overline{B\left(p,\frac{l}{2}\right)}\cup
 B\left(q,\frac{l}{2}\right).
 \end{eqnarray*}
Together with the fact that in this situation $B(p,\frac{l}{2})$,
$B(q,\frac{l}{2})$ are isometric to
$\mathscr{B}_{M^{\ast}}\left(p^{\ast},\frac{l}{2}\right)$, and also
using the symmetric property of the continuous function $k(t)$, the
rigidity conclusion of Theorem \ref{theo-1} follows. As explained in
Remark \ref{remark-added}, the regularity of the metric
(\ref{metric}) of $\overline{M^{\ast}}$ is highly depends on the
regularity and the vanishing property of even derivatives of the
warping function $f$. However, $f$ is determined by the system
(\ref{warpf}), or roughly speaking, by the radial Ricci curvature
bound $(n-1)k(t)$. Hence, the regularity of gluing two domains
$\overline{B(p,l/2)}$ and $B(q,l/2)$ would be assured by the
assumptions on $k(t)$. Besides, it is not hard to see that
$\overline{M^{\ast}}$ (together with the one-point compactification
topology at the point $\{l\}\times_{0}\mathbb{S}^{n-1}$) is
symmetric with respect to the submanifold
$\{\frac{l}{2}\}\times_{f(l/2)}\mathbb{B}^{n}$, where
$\mathbb{B}^{n}$ denotes the Euclidean $n$-ball whose boundary is
$\mathbb{S}^{n-1}$. The proof is finished.
\end{proof}

Once the main conclusion of Theorem \ref{theo-1} has been proved,
the two sphere theorems stated in Corollaries \ref{coro-1} and
\ref{coro-2} follow directly.

\section*{Acknowledgments}
\renewcommand{\thesection}{\arabic{section}}
\renewcommand{\theequation}{\thesection.\arabic{equation}}
\setcounter{equation}{0} \setcounter{maintheorem}{0}

This research was supported in part by the NSF of China (Grant Nos.
11801496 and 11926352), the Fok Ying-Tung Education Foundation
(China), the Key Project of Jiangxi Provincial Natural Science
Foundation (Grant No. 20252BAC250003), Hubei Key Laboratory of
Applied Mathematics (Hubei University), and Key Laboratory of
Intelligent Sensing System and Security (Hubei University), Ministry
of Education. The author is grateful to the anonymous referee for
careful reading and valuable comments such that the paper appears as
its present version.

\end{document}